\documentclass[11pt,a4paper,reqno]{amsart}

\usepackage{setspace}
\usepackage{fullpage}
\usepackage{datetime}
\usepackage{xcolor}

\pdftrailerid{}

\usepackage{enumitem}
\newcommand\rmlabel{\upshape({\itshape\roman*\,\/})}

\usepackage{amsmath}%
\usepackage{amsfonts}%
\usepackage{amssymb}%
\usepackage{graphicx}
\usepackage{mathrsfs}
\usepackage{hyperref}
\usepackage[margin=1in]{geometry}
\usepackage[normalem]{ulem}

\usepackage{cite}
\newtheorem{theorem}{Theorem}
\theoremstyle{plain}

\newtheorem{definition}[theorem]{Definition}

\newtheorem{lemma}[theorem]{Lemma}

\newtheorem{proposition}[theorem]{Proposition}

\numberwithin{equation}{section}
\numberwithin{theorem}{section}
\numberwithin{case}{section}

\numberwithin{subcase}{case}
\usepackage[abbrev,msc-links,backrefs]{amsrefs}
\usepackage{doi}

\renewcommand{\PrintDOI}[1]{\doi{#1}}

\usepackage{lineno}
\newcommand*\patchAmsMathEnvironmentForLineno[1]{%
\expandafter\let\csname old#1\expandafter\endcsname\csname #1\endcsname
\expandafter\let\csname oldend#1\expandafter\endcsname\csname end#1\endcsname
\renewenvironment{#1}%
{\linenomath\csname old#1\endcsname}%
{\csname oldend#1\endcsname\endlinenomath}}%
\newcommand*\patchBothAmsMathEnvironmentsForLineno[1]{%
\patchAmsMathEnvironmentForLineno{#1}%
\patchAmsMathEnvironmentForLineno{#1*}}%
\AtBeginDocument{%
\patchBothAmsMathEnvironmentsForLineno{equation}%
\patchBothAmsMathEnvironmentsForLineno{align}%
\patchBothAmsMathEnvironmentsForLineno{flalign}%
\patchBothAmsMathEnvironmentsForLineno{alignat}%
\patchBothAmsMathEnvironmentsForLineno{gather}%
\patchBothAmsMathEnvironmentsForLineno{multline}%
}

\def \A{\mathcal{A}}

\def \F{\mathcal{F}}

\def \Z{\mathcal{Z}}

\def \bbN{\mathbb{N}}

\def\a{\alpha}
\def\b{\beta}

\def \eps{\varepsilon}

\def \r{\gamma}
\def \l{\lambda}
\def\COMMENT#1{}
\let\COMMENT=\footnote

\let\subset\subseteq
\let\epsilon\varepsilon
\DeclareMathOperator{\ex}{{\rm ex}}

\begin{document}

\title{Factors in randomly perturbed hypergraphs}

\author[Y.~Chang]{Yulin Chang}
\address{(Y. Chang) School of Mathematics, Shandong University, Jinan,
  250100, China}
\email{ylchang93@163.com}

\author[J.~Han]{Jie Han}
\address{(J. Han) Department of Mathematics, University of Rhode
  Island, Kingston, RI, 02881, USA}
\email{jie\_han@uri.edu}

\author[Y.~Kohayakawa]{Yoshiharu Kohayakawa}
\address{(Y. Kohayakawa and G. O. Mota) Instituto de Matem\'atica e Estat\'{\i}stica, Universidade de
  S\~ao Paulo, S\~ao Paulo, Brazil}
\email{\{yoshi|mota\}@ime.usp.br}

\author[P.~Morris]{Patrick Morris}
\address{(P. Morris) Freie Universit\"at Berlin and Berlin
  Mathematical School, Berlin, Germany}
\email{pm0041@math.fu-berlin}

\author[G.~O.~Mota]{Guilherme Oliveira Mota}

\thanks{Part of this work was done while the first author was at the
  University of Rhode Island as a visiting student, partially
  supported by the Chinese Scholarship Council. The research of the
  second author was supported in part by Simons Foundation \#630884.
  The research of the third author was supported by CNPq
  (311412/2018-1, 423833/2018-9) and FAPESP (2018/04876-1,
  2019/13364-7).  The
  research of the fourth author was supported in part by the Deutsche
  Forschungsgemeinschaft (DFG, German Research Foundation) under
  Germany's Excellence Strategy - The Berlin Mathematics Research
  Center MATH+ (EXC-2046/1, project ID: 390685689).  The research of
  the fifth author was supported by CNPq (306620/2020-0,
  428385/2018-4) and FAPESP (2018/04876-1, 2019/13364-7).  This
  research was
  partially supported by CAPES (Finance Code 001).  FAPESP is the
  S\~ao Paulo Research Foundation.  CNPq is the National Council for
  Scientific and Technological Development of Brazil.
}

\begin{abstract}%
  We determine, up to a multiplicative constant, the optimal number of
  random edges that need to be added to a $k$-graph~$H$ with minimum
  vertex degree~$\Omega(n^{k-1})$ to ensure an $F$-factor with  high
  probability, for any~$F$ that belongs to a certain
  class~$\mathcal{F}$ of $k$-graphs, which includes, e.g., all
  $k$-partite $k$-graphs,  $K_4^{(3)-}$ and the Fano plane.
  In particular, taking~$F$ to be a single edge, this settles a
  problem of Krivelevich, Kwan and Sudakov [Combin.\ Probab.\
  Comput.~25 (2016), 909--927].
  We also address the case in which the host graph~$H$ is not dense,
  indicating that starting from certain such~$H$ is essentially the
  same as starting from an empty graph (namely, the purely random
  model).
\end{abstract}

\shortdate
\yyyymmdddate
\settimeformat{ampmtime}
\date{}
\maketitle

\section{Introduction}

\subsection{$F$-factors}
Given graphs $F$ and $G$, an \emph{$F$-tiling} of $G$ is a collection of vertex-disjoint copies of $F$ in $G$. An $F$-tiling of $G$ is called \emph{perfect} if it covers all the vertices of $G$.
A perfect $F$-tiling is also referred to as an \emph{$F$-factor} or a \emph{perfect $F$-packing}.  Note that a perfect matching corresponds
to a $K_2$-factor.

Kirkpatrick and Hell~\cite{HK} showed that the problem of deciding whether a graph $G$ has an $F$-factor is NP-complete if and only if $F$ has a component which contains three or more vertices.
Thus it is natural to ask for conditions that guarantee the existence of an $F$-factor in a graph~$G$, for such graphs~$F$.
The celebrated Hajnal--Szemer\'{e}di theorem~\cite{HS} states that every graph on $n$ vertices with minimum degree at least $(1 -1/r)n$ contains a $K_r$-factor.  For arbitrary graphs $F$, K\"{u}hn and Osthus~\cite{KO} determined, up to an additive constant, the minimum degree of a graph $G$ which ensures an $F$-factor in $G$, improving results of \cite{AY,KSS}.

Let $k\geq 2$ be given.
A \emph{$k$-uniform hypergraph}, or \emph{$k$-graph} for short, $H=(V,E)$ consists of a \emph{vertex set} $V$ and an \emph{edge set} $E\subseteq \binom{V}{k}$, where $\binom{V}{k}$ is the family of all the $k$-subsets of $V$.
If $E=\binom{V}{k}$, then $H$ is a \emph{complete $k$-graph}, denoted by $K_n^{(k)}$, where~$n=|V|$.
For any $d$-subset $S\subseteq V(H)$ with $1\leq d\leq k-1$, we define the \emph{degree} of $S$ to be $\deg_H(S):=\left|\left\{e\in E(H):\, S\subseteq e\right\}\right|$.
The \emph{minimum $d$-degree} $\delta_{d}(H)$ of $H$ is the minimum of $\deg_H(S)$ over all $d $-subsets $S$ of $V(H)$.
We often say that $\delta_{k-1}(H)$ is the \emph{minimum codegree} of $H$ and $\delta(H):=\delta_{1}(H)$ is the \emph{minimum vertex degree} of $H$.
We will be particularly concerned with \emph{dense} $k$-graphs, which can be defined using a minimum $d$-degree condition for any $1 \leq d \leq k-1$.
These degree conditions form a hierarchy, because if $\delta_d(H)=\Omega(n^{k-d})$ for some $1\leq d \leq k-1$, then $\delta_{d'}(H)=\Omega(n^{k-{d'}})$ for any $d'\leq d$.
Hence we have that requiring density with respect to minimum vertex degree is the weakest possible such condition, whilst requiring a linear minimum codegree is the strongest such condition.

The definition of $F$-factors extends naturally to hypergraphs, and it is very natural to study degree conditions that ensure the existence of $F$-factors in this generalised setting.
However, the problem becomes significantly harder, even for the simplest possible $k$-graphs~$F$.  For example, the minimum vertex degree threshold forcing the existence of a perfect matching in $k$-graphs remains unknown for $k\ge 6$.
For more results on factors in graphs and hypergraphs, we refer the reader to the excellent surveys~\cite{KO1,RR,Zhao}.

Another well-studied object in graph theory is the \emph{binomial random graph $G(n, p)$}, which has~$n$ vertices and each of its edges is present with probability~$p$, independently of all the other edges.
The \emph{binomial random $k$-graph}, which we denote by~$H^{(k)}(n,p)$, is defined analogously.
Determining the threshold for the appearance of $F$-factors in~$H^{(k)}(n,p)$ has been a notoriously hard problem, even for simple~$F$.
The threshold depends on a parameter of~$F$ defined as follows.
Given a ($k$-)graph $F$, we use $v_F$ and $e_F$ to denote, respectively, the number of vertices and edges in~$F$.
If~$F$ is a ($k$-)graph on at least two vertices, define
\begin{equation*}
    d^*(F):=\max\left\{\frac{e_{F'}}{v_{F'}-1}:\,F'\subseteq F,\
      v_{F'}\geq 2\right\}.
\end{equation*}
In an outstanding piece of work, Johansson, Kahn and Vu~\cite{JKV} made huge progress on this problem for both graphs and hypergraphs.
They conjectured that the threshold for the appearance of an $F$-factor in a binomial random ($k$-)graph is
\begin{equation}
  \label{eq:factor_threshold}
  \ell(n;F)n^{-1/d^*(F)},
\end{equation}
where $\ell(n;F)$ is an explicit polylogarithmic factor which depends on the structure of~$F$; see~\cite{JKV} for details.
Furthermore, they proved that the conjecture is true when we replace the~$\ell(n;F)$ term by some function which is $n^{o(1)}$ and they determined the exact threshold for all strictly balanced ($k$-)graphs~$F$, in which case one has $\ell(n;F)=(\log n)^{1/e_F}$.
The conjecture has now also been proven for the so-called non-vertex-balanced graphs~$F$, by Gerke and McDowell~\cite{GM}.
In this case, one has that $\ell(n;F)=1$.

\subsection{Randomly perturbed graphs}

In 2003, Bohman, Frieze and Martin~\cite{BFM} considered the problem of determining how many random edges one needs to add to a dense graph in order to guarantee that the resulting graph satisfies some structural property asymptotically almost surely (a.a.s.).
In recent years, the model has been extensively studied, for example exploring the Ramsey properties of such graphs~\cite{DMT,DT,KST,P}, and we now
know a wide range of results concerning embedding spanning subgraphs, such as bounded degree spanning trees and powers of Hamilton cycles
into randomly perturbed graphs (see, e.g.,~\cite{ADRRS,BTW,BDF,BHKMPP,BMPP,DRRS,HPT,JK,KKS,NT}).
In particular, Balogh, Treglown and Wagner~\cite{BTW} determined, for any fixed graph $F$, the number of random edges one needs to add to a graph $G$ of linear minimum degree to ensure that the resulting graph contains an $F$-factor a.a.s.

For two ($k$-)graphs $G$ and $G'$, denote by $G\cup G'$ the ($k$-)graph with vertex set $V(G)\cup V(G')$ and edge set $E(G)\cup E(G')$.

\begin{theorem}[Balogh, Treglown and Wagner~\cite{BTW}]
  \label{thm:BTW}
  Let $F$ be a fixed graph with $e_F>0$ and let $n\in \bbN$ be divisible by $v_F$.
  For every $\a > 0$, there exists $c = c(\a,F) > 0$ such that if $p\geq cn^{-1/d^*(F)}$ and $G$ is an $n$-vertex graph with minimum degree $\delta(G)\geq \a n$, then a.a.s. $G\cup G(n,p)$ contains an $F$-factor.
\end{theorem}

Comparing this result to~\eqref{eq:factor_threshold}, we see that, starting with a dense host graph $G$ instead of the empty graph reduces the number of random edges needed for forcing an $F$-factor by the multiplicative factor of~$\ell(n;F)$, which, e.g., is~$(\log n)^{1/e_F}$ in the case that~$F$ is strictly balanced~\cite{JKV} (in some cases~\cite{GM} there is no gain, as~$\ell(n; F)=1$ can happen).
We remark in passing that, in other contexts, the gain in the randomly perturbed model can even be polynomial in~$n$ (see, e.g., \cite{CHT,HZ,MM}).
This phenomenon is typical of the behaviour one observes in the randomly perturbed setting.

A natural question asked by Balogh, Treglown and Wagner \cite{BTW} is whether Theorem~\ref{thm:BTW} holds if we replace $\a n$ with a \emph{sublinear} term.
We note below that the answer to this question is negative in general (see Section~\ref{sec:question_from_BTW}).

\subsection{Randomly perturbed $k$-graphs}

As with graphs, randomly perturbed $k$-graphs are obtained by adding random ($k$-)edges to a certain $k$-graph.
Krivelevich, Kwan and Sudakov~\cite{KKS2} considered perfect matchings (and loose Hamilton cycles) under a minimum codegree condition in randomly perturbed $k$-graphs.
They proved that for every $k\geq 3$ and every $\a > 0$, there exists $\l > 0$ such that, if $H$ is a $k$-graph on $n\in k\bbN$ vertices with $\delta_{k-1}(H) \geq \a n$ and $p\geq \l n^{1-k}$, then a.a.s. the union $H\cup H^{(k)}(n,p)$ contains a perfect matching.
In addition, they also raised the analogous question for weaker minimum degree conditions.
For more results on randomly perturbed hypergraphs see \cite{BHKM, HZ,MM}.

In this paper we study $F$-factors in randomly perturbed hypergraphs.
In particular, we solve the aforementioned question of Krivelevich, Kwan and Sudakov~\cite{KKS2}.

\begin{theorem}
  \label{thm:matching}
  Let~$k\geq2$ be given and let $n\in k\bbN$.
  For every $\eps>0$, there is a constant $c = c(k, \eps) > 0$ such that if $p\geq cn^{1-k}$ and $H$ is an $n$-vertex $k$-graph with minimum vertex degree
  $\delta(H)\geq \eps\binom{n-1}{k-1}$, then a.a.s. $H\cup H^{(k)}(n,p)$ contains a perfect matching.
\end{theorem}

We shall derive Theorem \ref{thm:matching} from our result on general $F$-factors (see Theorem \ref{main} below).
To state our result we first introduce some notation.
For a $k$-graph $F$ and a vertex $v\in V(F)$, we define $F_v$ to be the spanning subgraph of $F$ consisting of the edges containing~$v$.
Let $\Lambda_{F,v}$ be the collection of $\a>0$ such that the following holds: for every $\eps>0$, there exist $\eps'>0$ and $n_0\in \bbN$ such that if $H$ is a $k$-graph with $n\geq n_0$ vertices and $\delta(H)\geq(\a+\eps)\binom{n-1}{k-1}$, then for every $w\in V(H)$ there are at least $\eps' n^{v_F-1}$
embeddings of~$F_v$ in~$H$ that map~$v$ to~$w$.
Let
\begin{equation}
  \label{eq:3}
  \a_F:=\min_{v\in V(F)} \inf\Lambda_{F,v}.
\end{equation}
(An alternative definition of~$\alpha_F$ is given in Section~\ref{sec:concluding}.)

We are now ready to state our result.

\begin{theorem}[Main result]
  \label{main}
  Let~$k\geq2$ be given.  Let $F$ be a $k$-graph with $e_F>0$ and let $n\in \bbN$ be divisible by $v_F$. For every $\eps>0$, there is a constant $c = c(k,\a_F,\eps) > 0$ such that if $p\geq cn^{-1/d^*(F)}$ and $H$ is an $n$-vertex $k$-graph with minimum vertex degree $\delta(H)\geq (\a_F+\eps)\binom{n-1}{k-1}$, then a.a.s. $H\cup H^{(k)}(n,p)$ contains an $F$-factor.
\end{theorem}

Theorem~\ref{main} shows that, starting with a $k$-graph~$H$ that is dense only in the sense of having a lower bound on its minimum vertex degree, the number of random edges needed to be added to~$H$ in order to force an $F$-factor with high probability is, for certain~$F$, reduced by a polylogarithmic factor (recall~\eqref{eq:factor_threshold}).

In the case in which~$\alpha_F=0$, the condition on the minimum degree of~$H$ as well as the lower bound on~$p$ in Theorem~\ref{main} are optimal.
This is made more precise in Section~\ref{sec:question_from_BTW}.

By the supersaturation result of Erd\H{o}s and Simonovits~\cite[Corollary 2]{Erdos}, we have~$\a_F=0$ if and only if there exists $v\in V(F)$ such that the link $(k-1)$-graph of $v$ is $(k-1)$-partite, which is the case when $F$ is a single edge, a $k$-partite $k$-graph, $K_4^{(3)-}$ (the $3$-graph with~$4$ vertices and~$3$ edges), the Fano plane, etc.
In particular,
\begin{itemize}
\item taking $k=2$, since $F_v$ is a star for any $F$, we derive $\alpha_F=0$ and recover Theorem~\ref{thm:BTW};
\item since $d^*(F)=1/(k-1)$ when~$F$ is a single $k$-edge, Theorem~\ref{main} reduces to Theorem~\ref{thm:matching}.
\end{itemize}

\subsection{Proof ideas and organisation}
To have an $F$-factor in~$H\cup H^{(k)}(n,p)$, we clearly have to be above the \emph{covering threshold}: every $w\in V(H)$ has to be covered by a copy of~$F$ in~$H\cup H^{(k)}(n,p)$.
The definition of~$\alpha_F$ and the minimum degree condition on~$H$ tells us that every~$w\in V(H)$ is covered by many copies of~$F_v$ in~$H$ (for the
`cheapest' $v\in V(F)$).
This suggests the following two-step strategy for covering a vertex $w\in V(H)$ by a copy of $F$ in~$H\cup H^{(k)}(n,p)$: we use~$H$ to cover~$w$ by a copy of $F_v$ and then use the random edges of~$H^{(k)}(n,p)$ to fill the missing edges to yield a copy of~$F$.

Our proof uses the ``absorption technique'' pioneered by R\"{o}dl, Ruci\'nski and Szemer\'{e}di~\cite{Rodl1}, combined with some results concerning binomial random hypergraphs.
The key step is to build absorbers for an arbitrary set of $v_F$ vertices and for this we use the two-step strategy mentioned above.

The rest of the paper is organised as follows.
In Section~\ref{s2} we prove some useful results concerning binomial random $k$-graphs.
In Section~\ref{sec:question_from_BTW} we discuss the question raised in~\cite{BTW} mentioned above and we discuss the optimality of Theorem~\ref{main} when~$\alpha_F=0$.
In Section~\ref{s3} we prove our absorbing lemma and the proof of Theorem~\ref{main} is given in Section~\ref{s4}.
Some remarks concerning the parameter~$\alpha_F$ are given in Section~\ref{sec:concluding}.

Throughout the rest of the paper, $k$~denotes an integer with~$k\geq2$ and will be omitted from the dependencies of constants in both the
results and proofs.
As usual, for an integer~$b$, let~$[b]=\{1,\dots,b\}$.
For simplicity, we omit floor and ceiling signs when they are not essential.

\section{Preliminaries}
\label{s2}

In this section we discuss some relevant results related to binomial random $k$-graphs and introduce some notation used throughout the paper.

Given a $k$-graph $H$ and a subset $S\subseteq V(H)$, we write $H[S]$ for the subgraph of $H$ induced by $S$.
We write $x\ll y\ll z$ to mean that we can choose constants from right to left, that is, there exist functions $f$ and $g$ such that, for any $z>0$, whenever $y\leq f(z)$ and $x\leq g(y)$, the subsequent statement holds.
Statements with more variables are defined similarly.
We assume that $n$ is sufficiently large, unless stated otherwise.

Consider the random $k$-graph $H^{(k)}(n, p)$ with $p=p(n)$.
Following~\cite{Janson}, for a fixed $k$-graph~$F$, we let
\[
  \Phi_F=\Phi_F(n,p)
  :=\min\left\{n^{v_{F'}}p^{e_{F'}}:\,F'\subseteq F,\ e_{F'}>0\right\}.
\]
We shall need the following result, which follows directly from \cite[Proposition~2.1]{BHKM} and Chebyshev's inequality.

\begin{proposition}[{\cite[Proposition 2.1]{BHKM}}]\label{prop:BHKM}
  Let $F$ be a $k$-graph with $s$ vertices and $f$ edges and let $H = H^{(k)}(n, p)$.
  Let $\A$ be a family of ordered $s$-subsets of $V = V(H)$.
  For each $A\in\A$, let~$I_A$ be the indicator random variable of the event that $A$ spans a labelled copy of $F$ in $H$.
  Let $X=\sum_{A\in \A}I_A$.
  Then $\mathbb{P}\left[X\geq 2\mathbb{E}(X)\right]\leq s!2^{2s}n^{2s}p^{2f}/(\mathbb{E}(X)^2\Phi_F)$.
\end{proposition}

The next proposition gives a lower bound for $\Phi_F$ for some
specific $k$-graphs $F$.
Given labelled $k$-graphs $F_1$ and $F_2$, we denote by $F_1\cup F_2$
the $k$-graph with vertex set $V(F_1)\cup V(F_2)$ and edge set
$E(F_1)\cup E(F_2)$.

\begin{proposition}\label{lem:F12}
  Let $F_1$ and $F_2$ be labelled $k$-graphs with $V(F_1)\cap V(F_2)=\{v\}$.
  Then $\Phi_{F_1\cup F_2}\geq \min\{\Phi_{F_1},\Phi_{F_2},\Phi_{F_1}\Phi_{F_2} n^{-1}\}$.
\end{proposition}

\begin{proof}
  Let $F'$ be a subgraph of $F_1\cup F_2$ with $v'$ vertices and $e'$ edges, where $e'>0$.
  Let $v_i:=|V(F')\cap V(F_i)|$ and $e_i:=|E(F')\cap E(F_i)|$ for $i\in\{1,2\}$.
  If $V(F')\subseteq V(F_i)$ for some $i\in\{1,2\}$, then
  \[
    n^{v'}p^{e'}=n^{v_i}p^{e_i}\geq \Phi_{F_i};
  \]
  otherwise, we have $V(F')\cap V(F_i)\neq \emptyset$ for $i\in\{1,2\}$.
  Since $v'\geq v_1 + v_2 -1$ and~$e'=e_1+e_2$, we have
  \[
    n^{v'}p^{e'}\geq
    n^{v_1+v_2-1}p^{e_1+e_2}=(n^{v_1}p^{e_1})(n^{v_2}p^{e_2})n^{-1}\geq
    \Phi_{F_1}\Phi_{F_2} n^{-1}.
  \]
  This shows that $\Phi_{F_1\cup F_2}\geq \min\{\Phi_{F_1},\Phi_{F_2},\Phi_{F_1}\Phi_{F_2} n^{-1}\}$.
\end{proof}

Let $F$ be a labelled $k$-graph with $b$ vertices.
Let $A$ be a $k$-graph composed of copies $F'$, $F_1,\dots, F_b$ of
$F$ such that $V(F_1),\ldots, V(F_b)$ are pairwise disjoint
and~$|V(F')\cap V(F_i)|=1$ for every~$i\in[b]$.
Denote by $\A(F)$ the collection of all $k$-graphs~$A$ defined this way.
The next result bounds~$\Phi_A$ for every $A\in\A(F)$.

\begin{lemma}
  \label{lem:PhiA}
  Let $c\geq1$ be given and let $F$ be a labelled $k$-graph with $b$ vertices.
  If $p=p(n)$ is such that $\Phi_F\geq cn$, then $\Phi_A\geq cn$ for every $A\in\A(F)$.
\end{lemma}

\begin{proof}
  Fix~$A\in \A(F)$, and let $F'$, $F_1,\dots,F_b$ be as in the
  definition of~$\A(F)$ above.  By assumption we have
  $\Phi_{F'}=\Phi_{F_1}=\dots=\Phi_{F_b}\geq cn$.
  Let~$A_i=F_i\cup\dots\cup F_1\cup F'$  for every~$i\in[b]$.  Note
  that~$A=A_b$.  As~$c\geq1$, Proposition~\ref{lem:F12} tells us that
  $\Phi_{F_1\cup F'}\geq cn$.  Similarly, a simple induction tells us
  that~$\Phi_{A_i}\geq cn$ for every~$2\leq i\leq b$ and our lemma
  follows.
\end{proof}

The following simple lemma provides a lower bound on $p$ which ensures that the condition for $p$ in Lemma \ref{lem:PhiA} holds.

\begin{lemma}
  \label{lem:PhiF}
  Let $c\geq1$ be given and let $F$ be a labelled $k$-graph.
  If $p=p(n)\geq cn^{-1/d^*(F)}$, then $\Phi_F\geq cn$.
\end{lemma}

\begin{proof}
  Let $F'$ be an arbitrary subgraph of $F$ with $v'$ vertices and $e'>0$ edges.
  Note that, from the definition of $d^*(F)$, we have $d^*(F)\geq e'/(v'-1)$.
  Then, we have
  \[
    n^{v'}p^{e'}\geq n^{v'}(cn^{-1/d^*(F)})^{e'}\geq
    n^{v'}(cn^{-(v'-1)/e'})^{e'}=c^{e'}n\geq cn. \qedhere
  \]
\end{proof}

We shall use the following result of~\cite{BHKM}, which ensures the existence of an almost $F$-factor in the random hypergraph $H^{(k)}(n,p)$, for any fixed $k$-graph $F$ and $p=p(n)$ such that $\Phi_{F}\geq cn$.

\begin{lemma}[\cite{BHKM}, part of Lemma 2.2]
  \label{lem:BHM}
  Let $F$ be a labelled $k$-graph with $b$ vertices and $f$ edges.
  Suppose $1/n\ll1/c\ll \l,\,1/b,\,1/f$.
  Let $V$ be an $n$-vertex set, and let $\F$ be a family of $\l n^b$ ordered $b$-subsets of $V$.
  If $p=p(n)$ is such that $\Phi_{F}(n,p)\geq cn$, then the following properties hold for the binomial random $k$-graph $H = H^{(k)}(n, p)$ on $V$.
  \begin{enumerate}[label=\rmlabel]
  \item \label{BHMi} With probability at least $1-\exp(-n)$, every induced subgraph of $H$ of order $\l n$ contains a copy of $F$.
  \item \label{BHMii} With probability at least $1-\exp(-n)$, there are at least $(\l/2)n^bp^f$ ordered $b$-sets in $\F$ that span labelled copies of $F$.
\end{enumerate}
\end{lemma}

\section{Optimality of Theorem~\ref{main} and a question
  from~\cite{BTW}}
\label{sec:question_from_BTW}

\subsection{On the optimality of Theorem~\ref{main} when~$\alpha_F=0$}
Let us show that, in Theorem~\ref{main}, we need to have $p=\Omega(n^{-1/d^*(F)})$ when $\a_F=0$ (this has already been observed in the case~$k=2$ in~\cite{BTW}).
In the next section, we compare our model with the purely random model
in the context of Theorem~\ref{main}.
We show that we need a bound of the form $\delta(H)=\Omega(n^{k-1})$ and the fact that~$F$ is strictly balanced in order to see a substantial difference in the behaviour of the two models.

Fix a $k$-graph $F$ with $\a_F=0$.
We exhibit a sequence of $n$-vertex $k$-graphs $H_n$ with minimum vertex degree $\Omega(n^{k-1})$ such that, if $p \ll n^{-1/d^*(F)}$, then a.a.s.\ $H_n\cup H^{(k)}(n,p)$ does not contain an $F$-factor.
We shall need the following result.

\begin{proposition}
  \label{prop2}
  Let $F$ be a $k$-graph with $e_F>0$.
  For every $\theta>0$, there is a positive constant $c = c(\theta,F)$ such that if $p\leq cn^{-1/d^*(F)}$, then
  \[ \lim_{n\to\infty}\mathbb{P}\left[H^{(k)}(n,p)
      \mbox{ contains an $F$-tiling covering at least } \theta n \mbox{
        vertices}\right] = 0.
  \]
\end{proposition}
\begin{proof}
  Let $F$ be a $k$-graph with $b$ vertices and $f$ edges and let $\theta>0$ be given.
  Recall that $d^*(F):=\max\left\{e_{F'}/(v_{F'}-1):\,F'\subseteq F,\ v_{F'}\geq 2\right\}$.
  Suppose $J$ is a subgraph of $F$ that achieves the maximum in the definition of~$d^*(F)$.
  Suppose~$J$ has $s$ vertices and $j$ edges.
  Then $d^*(F)=d^*(J)=j/(s-1)$.
  Let~$c=(\theta/(2b))^{1/j}$ and $p= cn^{-1/d^*(F)}$.
  Note that it suffices to prove that
  \begin{equation}
    \label{eq:2}
    \lim_{n\to\infty}\mathbb{P}\left[H^{(k)}(n,p)
      \mbox{ contains a $J$-tiling with at least }\theta n/b
      \mbox{ members}\right] = 0.
  \end{equation}

  Let $H:=H^{(k)}(n,p)$.
  We claim that $\Phi_J\geq c^jn$.
  Indeed, let $J'$ be a subgraph of $J$ with $t$ vertices and $j'$ edges.
  By the choice of $J$, we have $(s-1)/j\leq (t-1)/j'$. Note also that $c\leq 1$. Thus,
  \[
    n^tp^{j'}=n^t\left(cn^{-1/d^*(F)}\right)^{j'} =
    n^{t}\left(cn^{-(s-1)/j}\right)^{j'} =
    \left(c^{j'}n\right)\left(n^{(t-1)/j'-(s-1)/j}\right)^{j'}\geq
    c^{j'}n\geq c^jn.
  \]
  Hence $\Phi_J\geq c^jn$.

  Let $\A$ be the family of all ordered $s$-sets of vertices from $V(H)$.
  For each $A\in\A$, let~$I_A$ be the indicator random variable of the event that $A$ spans a labelled copy of $J$ in~$H$.
  Let $X: =\sum_{A\in \A}I_A$.
  Then $\mu:=\mathbb{E}(X)=n(n-1)\cdots(n-s+1)p^j\geq n^sp^j/2$.
  From Proposition~\ref{prop:BHKM}, we have
  \[
    \mathbb{P}\left[X\geq 2\mu\right]\leq
    \frac{s!2^{2s}n^{2s}p^{2j}}{\mu^2\Phi_J}\leq
    \frac{s!2^{2s+2}n^{2s}p^{2j}}{n^{2s}p^{2j}
      \Phi_J}=O\left(\frac{1}{\Phi_J}\right)=o(1).
  \]
  Note that $n^sp^j=n^s(cn^{-1/d^*(F)})^j=c^jn=\theta n/(2b)$.
  Hence a.a.s. $X<2\mu\leq2n^sp^j=\theta n/b$, which clearly implies~\eqref{eq:2}, as required.
\end{proof}

Let us now define some $k$-graphs~$H_n$ on~$n$ vertices with~$\delta(H_n)=\Omega(n^{k-1})$.

\begin{definition}\label{def:Hn}
  \emph{Let $H_n:=H_n(A,B,\eta)$ be the $n$-vertex $k$-graph whose vertex set can be partitioned into two disjoint sets $V=A\cup B$ such that $|A|=\eta n$, and whose edges are all the $k$-sets in $V$ that intersect $A$.}
\end{definition}

Note that~$\delta(H_n)=\binom{n-1}{k-1}-\binom{(1-\eta)n-1}{k-1}=(1+o(1))\big(1-(1-\eta)^{k-1}\big){n-1\choose k-1}$ where~$o(1)\to0$ as~$n\to\infty$. Since $1-(1-\eta)^{k-1}\ge \eta/2$ by induction on $k$, we have
$\delta(H_n)\geq(\eta/(2(k-1)!))n^{k-1}$ if~$n$ is sufficiently large.

Consider any $k$-graph $F$ with $b$ vertices and at least one edge and let $0<\eta<1/b$ be arbitrary.
Fix $\theta:=1-b\eta>0$ and let $c = c(\theta,F)$ be given by Proposition~\ref{prop2} applied with~$k$, $F$ and~$\theta$.
Consider $n\in \bbN$ sufficiently large, and let $H_n:=H_n(A,B,\eta)$ be as in Definition~\ref{def:Hn}.
Let $H'_n:=H^{(k)}(n,p)$, where $p\leq cn^{-1/d^*(F)}$, and note that if $H_n\cup H'_n$ contains an $F$-factor, then $H'_n[B]$ must contain an
$F$-tiling covering at least $|B|-(b-1)|A|=(1-\eta)n-\eta n(b-1)=\theta n$ vertices.
However, Proposition~\ref{prop2} implies that a.a.s.\ there is no such $F$-tiling in $H'_n[B]$, and thus, a.a.s.\ $H_n\cup H'_n$ does not contain an $F$-factor.
This shows that we must require~$p=\Omega(n^{-1/d^*(F)})$ in Theorem~\ref{main} when $\alpha_F=0$.

\subsection{A question of Balogh, Treglown and Wagner}
A natural question asked by Balogh, Treglown and Wagner \cite{BTW} is whether Theorem~\ref{thm:BTW} holds if we replace $\a n$ with a \emph{sublinear} term.
We note here that the answer to this question is negative for graphs and also for $k$-graphs with minimum vertex degree $o(n^{k-1})$ in general, as we now demonstrate when $F$ is a single $k$-edge.
Indeed, let $H_n:=H_n(A,B,\eta)$ be as in Definition~\ref{def:Hn}, where $\eta=1/(3k\omega)$ with $\omega=\omega(n)\to\infty$ as $n\to\infty$, and say~$\omega=o(n)$.
Then we have
$\delta(H_n)\geq(\eta/(2(k-1)!))n^{k-1}=\Omega(n^{k-1}/\omega)$.
Let $p=(1/2)\binom{n-1}{k-1}^{-1}\ln\omega$.
Next we show that a.a.s.~$H_n\cup H^{(k)}(n,p)$ does not contain a perfect matching.
To see this, let~$X$ denote the number of isolated vertices in $H'_n=H^{(k)}(n,p)$.
Then
\[
  \mathbb{E}(X)=n(1-p)^{\binom{n-1}{k-1}}
  \geq ne^{-2p\binom{n-1}{k-1}}= n/\omega.
\]
Computing the second moment yields that a.a.s.~there are at least $n/(2\omega)$ isolated vertices in~$H'_n$.
On the other hand, a matching of~$H_n$ can cover at most~$k|A|=n/(3\omega)$ vertices.
It follows that~$H_n\cup H_n'$ does not have a perfect matching a.a.s.

The discussion above tells us that, in the randomly perturbed model, if the minimum degree condition on the host graph~$H$ on~$n$ vertices is only that~$\delta(H)\geq n^{k-1}/\omega$ and~$\omega\to\infty$ as~$n\to\infty$, then we need~$p\gg n^{-k+1}$ for~$H\cup H^{(k)}(n,p)$ to contain a perfect matching a.a.s.
Furthermore, if the condition is~$\delta(H)\geq n^{k-1-\epsilon}$ for some~$\epsilon>0$, then~$p=\Omega(n^{-k+1}\log n)$, which matches the purely random
threshold for a perfect matching~\eqref{eq:factor_threshold}, as $\ell(n;F)=\log n$ in this case.

The argument above applies to other ($k$-)graphs $F$, such as strictly balanced ones.
We omit the details.

\section{The Absorbing Lemma}
\label{s3}
In this section we prove our absorbing lemma, which is the key result in the proof of Theorem~\ref{main}.
Let us first give the following definition concerning the absorption method.

\begin{definition}\label{def:ab}
  \emph{Let $F$ be a $k$-graph with $b$ vertices.
  Suppose $H$ is a $k$-graph with vertex set~$V$ and let $S\subseteq
  V$ with $|S|= b$ be given.
  We call $\emptyset\neq A\subseteq V\setminus S$ with $|A|\in b \bbN$
  an \emph{$(S,F)$-absorber}, or
  \emph{$S$-absorber} for short, if both $H[A]$ and $H[A\cup S]$ contain
  $F$-factors.}
\end{definition}

In order to obtain an absorbing set, we need the following result
from~\cite{Nena}, which follows from~\cite[Lemma~10.7]{Mont2} (see
also~\cite[Lemma~2.8]{Mont}).

\begin{lemma}[\cite{Nena}, Lemma~2.3]
  \label{lem:NP}
  Let~$0<\b\leq 1$ be given. There exists $m_0$ such that the
  following holds for every $m\geq m_0$. There exists a bipartite
  graph $B$ with vertex classes $X_m\cup Y_m$ and~$Z_m$ and maximum
  degree\footnote{In~\cite{Nena}, this lemma is stated with $\Delta\leq40$ instead of
  $\Delta\leq100$.} $\Delta(B)\leq 100$, such that $|X_m|=m+\b m$, $|Y_m| =2m$
  and $|Z_m| = 3m$, and for every subset $X_m'\subseteq X_m$ with
  $|X_m'|=m$, the induced graph $B[X_m'\cup Y_m, Z_m]$ contains a
  perfect matching.
\end{lemma}

We use Lemma~\ref{lem:NP} to prove the following absorbing lemma for
hypergraphs, obtaining a sufficient condition for the existence of an
absorbing set.  This lemma extends to $k$-graphs a result for graphs
($k=2$) obtained in~\cite[Lemma 2.2]{Nena}.  For later convenience, we
state this lemma in a form slightly more general than we need here
(in this paper, we apply this lemma with~$V_0=\emptyset$).

\begin{lemma}[Absorbing Lemma]
  \label{lem:abs}
  For every positive integer $b$ and every $\r>0$ there exists $\xi$
  such that the following holds for every
  $0<\varepsilon<\min\{1/3, \r/2\}$ and every sufficiently
  large~$n$.  Let~$F$ be a $k$-graph with~$b$ vertices.  Suppose~$H$
  is an $n$-vertex $k$-graph and there exists $V_0\subset V(H)$
  of size at most~$\varepsilon n$ such that {\rm(\textit{i})}~for
  every $b$-subset~$S$
  of $V(H)\setminus V_0$, there are at least~$\r n$ vertex-disjoint
  $S$-absorbers and {\rm(\textit{ii})}~for every~$v\in V(H)$, there
  are at least~$\r n$ copies of~$F$ containing~$v$ that are pairwise
  vertex-disjoint, except for the vertex~$v$.  Then~$H$ contains a subset
  $A\subseteq V(H)$ of size at most~$\r n$ such that, for every
  $R\subseteq V(H)\setminus A$ with $|R|\leq \xi n$ such that~$b$
  divides $|A|+|R|$, the $k$-graph $H[A\cup R]$ contains an
  $F$-factor.
\end{lemma}

We refer to the sets~$A$ as in the lemma above as \emph{absorbing sets}.

\begin{proof}[Proof of Lemma~\ref{lem:abs}]
  For every~$v\in V(H)$, let~$\F_v$ be a family of~$\r n$ pairwise
  disjoint $(b-1)$-sets of vertices given by the copies of~$F-v$
  in~$H-v$ that hypothesis~(\textit{ii}) gives.

  Define the constants $q=\r/(1300b^2)$ and $\b=q^{b-1}\r/8$, and put~$\xi=\b q/(2(1+\b)(b-1))$.
  We will obtain a subset $X\subseteq V(H)\setminus V_0$ with $qn/2\leq|X|\leq 2qn$ such that for each $v\in V (H)$ there are at least $q^{b-1}|\F_v|/4$ sets from $\F_v$ contained in $X$.
  Consider a subset $X$ of~$V(H)\setminus V_0$ obtained by picking each vertex of $V(H)\setminus V_0$ with probability~$q$, independently of all the other vertices.
  Since $2qn/3\leq \mathbb{E}[|X|]=q(n-|V_0|)\leq qn$ tends to infinity as $n$ increases, by Chernoff's inequality, a.a.s.\ we have $qn/2\leq|X|\leq 2qn$.

  For every $v\in V(H)$, let $X_v$ denote the number of the sets from $\F_v$ contained in $X$.
  Clearly, $\mu:=\mathbb{E}[X_v]\geq q^{b-1}(|\F_v|-\varepsilon n)\geq q^{b-1}\r n/2$.
  Since $H$ has $n$ vertices, by using the union bound and Chernoff's inequality (see e.g. \cite[Theorem 2.1]{Janson}), we have
  \[
    \mathbb{P}\left[\,\text{there is}\, v\in V(H)\, \text{with}\, X_v<
      \frac{\mu}{2}\,\right] \leq
    n\exp\left(-\frac{(\mu/2)^2}{2\mu}\right)\leq
    n\exp\left(-\frac{q^{b-1}\r}{16} n\right)=o(1).
  \]

  Therefore, there is $X\subset V(H)\setminus V_0$ such that $qn/2\leq|X|\leq 2qn$ and such that, for each $v\in V (H)$, there are at least $\mu/2\geq q^{b-1}\r n/4$ sets from $\F_v$ contained in $X$.
  For each $v\in V(H)$, denote by $\F_v'$ the collection of members of~$\F_v$ that are completely contained in~$X$.
  We have $|\F_v'|\geq q^{b-1}\r n/4$.

  Let $m = |X|/(1+\b)$ and let $B$ be the bipartite graph given by Lemma \ref{lem:NP} with vertex classes $X_m\cup Y_m$ and $Z_m$.
  Choose arbitrarily vertex-disjoint subsets $Y,\,Z\subseteq V (H)\setminus (X\cup V_0)$ with $|Y | = 2m$ and $|Z| = 3m(b-1)$.
  Now partition~$Z$ arbitrarily into $(b-1)$-subsets $\Z = \{Z_i\}_{i\in[3m]}$ and fix bijections $\phi_1: X_m\cup Y_m \rightarrow X\cup Y$ and
  $\phi_2: Z_m \rightarrow \Z$ such that $\phi_1(X_m) = X$ and $\phi_1(Y_m) = Y$.

  We claim that there exists a family $\{A_e\}_{e\in E(B)}$ of pairwise vertex-disjoint $b^2$-subsets of $V (H)\setminus(X\cup Y\cup Z)$ such that for every $e = \{w_1,w_2\} \in E(B)$ with $w_1\in X_m\cup Y_m$ and $w_2\in Z_m$, the set $A_e$ is a $(\{\phi_1(w_1)\}\cup \phi_2(w_2))$-absorber.
  The idea here is to greedily choose such absorbers one by one for each $e\in E(B)$.
  Indeed, suppose we have already found appropriate subsets for all the edges from $E'\subset E(B)$ with~$E'\neq E(B)$.
  Note that $qn/(2+2\b)\leq m\leq 2qn/(1+\b)$ and $\Delta(B)\leq 100$.
  Therefore, since $|Z_m|\leq 3m$ we have
  \begin{equation}
    \label{eq:1}
    |X|+|Y|+|Z|+\bigg|\bigcup_{e\in E'}A_e\bigg| \leq
    4m+3m(b-1)+b^2|E'|<4bm+100b^2|Z_m| \leq \r n/2.
  \end{equation}
  Choose arbitrarily $e = \{w_1,w_2\} \in E(B)\setminus E'$ and let $S_e=\{\phi_1(w_1)\}\cup \phi_2(w_2)$.
  Note that $|S_e|=b$ and recall that there are at least $\r n$ vertex-disjoint $S_e$-absorbers in~$H$.
  Denote by~$\F_{e}$ the set of these vertex-disjoint $S_e$-absorbers.
  Since each vertex in $X\cup Y\cup Z\cup\bigcup_{e\in E'}A_e$ is in at most one $S_e$-absorber in $\F_{e}$, the number of $S_e$-absorbers disjoint from $X\cup Y\cup Z\cup\bigcup_{e\in E'}A_e$ is at least $\r n-\r n/2=\r n/2>0$ (recall~\eqref{eq:1}).
  Therefore, we can apply the described procedure repeatedly until we obtain $\{A_e\}_{e\in E(B)}$.

  Let $A=X\cup Y\cup Z\cup\bigcup_{e\in E(B)}A_e$.
  Then $|A|\leq \gamma n$.
  We claim that $A$ is an absorbing set as desired.
  Consider any subset $R\subseteq V(H)\setminus A$ such that $|R|+|A|\in b\bbN$ and $0\leq|R|\leq \xi n$.
  Suppose for a moment that $Q\subset X$ is such that $|Q|=\b m$.
  Then, setting $X''=X\setminus Q$ and noting that $|X''|= m$, from Lemma~\ref{lem:NP}, considering $X_m'=\phi_1^{-1}(X'')$, we see that there is a perfect matching $M$ in $B$ between $X_m'\cup Y_m$ and $Z_m$.
  For each edge $e = \{w_1,w_2\}\in M$ take an $F$-factor in $H[\{\phi_1(w_1)\}\cup\phi_2(w_2)\cup A_e]$ and for each $e \in E(B)\setminus M$ take an $F$-factor in $H[A_e]$.
  All together, this gives an $F$-factor of $H[A\setminus Q]$.
  Thus, if we can find a $Q\subset X$ with $|Q|=\b m$ and such that there is an $F$-factor in $H[Q\cup R]$, then together with the above argument we
  get an $F$-factor in $H[A \cup R]$, as required.

  To find a set~$Q$ as above, note first that $(b-1)|R|\leq (b-1)\xi n=\b qn/(2+2\b)\leq \b m$.
  We claim that $\b m-(b-1)|R|\in b\bbN$.
  Indeed, taking disjoint subsets $Q_1,\,Q_2\subset X $ such that $|Q_1|=\b m-(b-1)|R|$ and $|Q_2|=(b-1)|R|$, we certainly have that $|A|+|R|,\,|Q_2|+|R|\in b \bbN$ and, by virtue of the existence of an $F$-factor in $H[A\setminus (Q_1 \cup Q_2)]$, we have that $|A|-|Q_1|-|Q_2|\in b \bbN$.
  Thus $|Q_1|=\b m-(b-1)|R|\in b\bbN$ also.
  Now we take an arbitrary subset $X'\subseteq X$ with $|X'|=\left(\b m-(b-1)|R|\right)/b$.
  Next we find a subset $A_v\in \F_v'$ for each $v\in R\cup X'$ such that all these subsets are pairwise vertex-disjoint and do not contain any vertex of~$X'$.
  In fact, we can choose such subsets greedily since $|X'|+(b-1)|X'|+(b-1)|R|=\b m$ and
  \[
    |\F_v'|\geq q^{b-1}|\F_v|/4\geq q^{b-1}\r n/4\geq 2\b n\geq 2\b m.
  \]
  Since~$A_v\in\F_v'$ for every~$v\in R\cup X'$, there is an $F$-factor in~$H[Q\cup R]$ if we set $Q:= X'\cup\bigcup_{v\in R\cup X'}A_v\subset X$.
  Also, note that $|Q| =\b m$ as required.
  This proves that $A$ is indeed the desired absorbing set.
\end{proof}

Now we introduce the absorbers that we use in this paper.

\begin{definition}
  \label{def:abs}
  \emph{Let $F$ be a $k$-graph with $b$ vertices and let $H$ be a $k$-graph with $n$ vertices.
  Let $S\subset V(H)$ with~$|S|=b$ be given.
  Suppose $A\subseteq V(H)\setminus S$ is the union of some pairwise disjoint $b$-sets $A_1,\dots,A_b$, say $A_i=\{v_i^1,\dots,v_i^b\}$ for each $i\in[b]$.
  Let $A_{b+1}=\{v_1^b,\dots,v_b^b\}$.
  We call $A$ a \emph{simple $S$-absorber} if for some labelling of the elements of~$S$, say $S=\{s_1,\dots,s_b\}$, both $H[A_i]$ and $H[\{s_i\}\cup(A_i\setminus\{v_i^b\})]$ contain copies of $F$ for all $i\in[b]$ and, furthermore, $H[A_{b+1}]$ contains a copy of $F$.}
\end{definition}

Clearly, simple $S$-absorbers are $S$-absorbers.
In the next lemma we prove that with appropriate conditions on~$p$ and on the minimum vertex degree of $H$, a.a.s.\ $H\cup H^{(k)}(n,p)$ contains many
simple $S$-absorbers for any $b$-subset~$S$ of~$V(H)$.
This result allows us to apply Lemma~\ref{lem:abs}.
Recall the definition of $\alpha_F$ from Theorem~\ref{main}.

\begin{lemma}
  \label{lem:absa}
  Let $F$ be a $k$-graph with $b$ vertices and let $0<\r\ll\eps,\,1/b$.
  Then there is a constant $c = c(\a_F,\eps) > 0$ such that the following holds for any $k$-graph~$H$ with~$n$ vertices and minimum degree $\delta(H)\geq (\a_F+\eps)\binom{n-1}{k-1}$.
  If $p\geq cn^{-1/d^*(F)}$, then a.a.s., for every $b$-subset~$S$
  of~$V(H)$, there are at least $\r n$ vertex-disjoint simple
  $S$-absorbers in $H\cup H^{(k)}(n,p)$.
\end{lemma}

\begin{proof}
  Let $F$ be a $k$-graph on~$b$ vertices and $f$ edges.
  Fix constants $0<\r\ll\eps,\,1/b$.
  Let $\eps_0=\eps_0(\eps/2)$ be the constant given in the definition of $\a_F$ (applied with $\eps/2$) and write $\eps'=\eps_0/2$.
  Furthermore, let $1/n\ll1/c \ll\eps',\,1/b,\,1/f$ and let $H$ be as in the statement of the lemma.
  Let~$V:=V(H)$.

  Suppose $p\geq cn^{-1/d^*(F)}$.
  Fix a $b$-subset
  $S=\{s_1,\dots,s_b\}$ of $V$.
  To find $\r n$ vertex-disjoint simple $S$-absorbers in $V\setminus S$ in $H\cup H^{(k)}(n,p)$, it suffices to show that for every subset $W\subseteq V\setminus S$ with $|W|=b^2(\r n-1)$, there is a simple $S$-absorber in $V\setminus (S\cup W)$.
  Fix such a $W$ and let $H'=H[V\setminus W]$ and $n'=n-|W| \ge (1-b^2\gamma)n$.
  Since $\delta(H)\geq (\a_F+\eps)\binom{n-1}{k-1}$ and $\r\ll \eps,\,1/b$, we have that
 \[
    \delta(H')\geq\delta(H)-|W|\binom{n}{k-2}
    \geq (\a_F+\eps)\binom{n-1}{k-1}-b^2\r n\binom{n}{k-2}
    \geq (\a_F+\eps/2)\binom{n'-1}{k-1}.
  \]
  Since $n'$ is large enough, by the definition of $\a_F$ and the choice of $\eps_0$, there is some $v^*\in V(F)$ such that for every $w\in V(H')$, there are at least $\eps_0 (n')^{b-1}$ embeddings of $F_{v^*}$ in $H'$ that map $v^*$ to~$w$.
  Among such embeddings, at least $\eps_0 (n')^{b-1} - b (n')^{b-2} \ge \eps' n^{b-1}$ of them do not intersect~$S$.

  Let $\F_S$ be a family of $b^2$-subsets $B\subseteq V'$ with $B = \bigcup_{i\in[b]}B_i$, where $B_i=\{v_i^1,\dots,v_i^b\}$ and $B_i\cap B_j=\emptyset$ for $i\neq j$, such that $H[(B_i\setminus\{v_i^b\})\cup\{s_i\}]$ contains a copy of $F_{v^*}$ for each $i\in[b]$.
  We can suppose that $|\F_S|\geq(\eps' n^{b-1})^bn^b -
  \binom{b^2}{2}n^{b^2-1} \geq (\eps')^bn^{b^2}/2$: greedily
  choose~$B_i$ for each~$i$ separately and then ignore sets of choices
  such that the~$B_i$ are not pairwise disjoint.

  Now we find a simple $S$-absorber in $V'$ by adding $H^{(k)}(n,p)$
  to fill the missing edges in some~$B\in\F_S$.  We shall make use of
  Lemma~\ref{lem:BHM}~\ref{BHMii} and Lemma \ref{lem:PhiA} to prove
  that this ``filling'' of some~$B$ happens with very high
  probability.  Fix~$B\in\F_S$ and let the notation be as above: in
  particular, let $B=\bigcup_{i\in[b]}B_i$ with
  $B_i=\{v_i^1,\dots,v_i^b\}$ for all~$i\in[b]$.  We now
  define~$A_B\in\A(F)$ with~$V(A_B)=B$ such that,
  if~$A_B\subset H^{(k)}(n,p)$, then~$B$ becomes a simple $S$-absorber
  in~$H\cup A_B\subset H\cup H^{(k)}(n,p)$.  To define~$A_B$, we first
  require that~$A_B[\{v_1^b,\dots,v_b^b\}]$ should be a copy of~$F$,
  and we further impose that~$A_B[B_i]$ should be a copy of~$F$ for
  every~$i\in[b]$ with the additional requirement that~$v_i^b$ should
  play the r\^ole of~$v^*$ in~$A_B[B_i]=F$.  (The conditions we have
  just described do not necessarily define~$A_B$ uniquely; we
  let~$A_B$ denote one such $k$-graph, selected arbitrarily.)  A
  moment's thought now shows that, because~$B\in\F_S$, we have that
  $(H\cup A_B)[\{s_i\}\cup(B_i\setminus\{v_i^b\})]$ contains a copy
  of~$F$ for every~$i\in[b]$.  Since~$(H\cup A_B)[B_i]$ ($i\in[b]$)
  and~$(H\cup A_B)[\{v_1^b,\dots,v_b^b\}]$ contain a copy of~$F$, we
  conclude that~$B$ is a simple $S$-absorber in~$H\cup A_B$, as
  required.

  Now, owing to Lemmas~\ref{lem:PhiA} and~\ref{lem:PhiF}, we have
  $\Phi_{A_B}\geq cn$.
  Let~$X_S$ count the number of~$B\in\F_S$ such that~$A_B\subset
  H^{(k)}(n,p)$.
  Applying Lemma~\ref{lem:BHM}~\ref{BHMii} we have
  \[
    \mathbb{P} [X_S\geq (\eps')^bn^{b^2}p^{(b+1)f}/4]
    \geq 1-\exp(-n),
  \]
  which implies that $\mathbb{P}[X_S=0]\leq \exp(-n)$.
  Since there are $\binom{n}{b}$ possible choices for $S$ and at most~$2^n$ possible choices for $W$, we have
  \[
    \mathbb{P}\left[\,\text{there are } S\in \binom{V}b
      \text{ and } W\subseteq V \text{ such that } X_S=0\right]
    \leq \binom{n}{b}2^n\exp(-n)=o(1).
  \]
  Thus, a.a.s.\ there is a simple $S$-absorber in $V'$ for every $S\in\binom{V}{b}$ and every $W\subseteq V\setminus S$ with $|W|=b^2(\r n-1)$.
  This concludes the proof.
\end{proof}

\section{Proof of the main theorem}
\label{s4}
In this section we use Lemmas~\ref{lem:abs},~\ref{lem:absa} and~\ref{lem:BHM}~\ref{BHMi} to prove Theorem~\ref{main}.

\begin{proof}[Proof of Theorem~\ref{main}]
  Suppose $F$ is a $k$-graph with~$b$ vertices and~$f$ edges and let $n\in b\bbN$ be sufficiently large.
  Let $\eps>0$ be given and suppose $1/n\ll\r\ll\eps,\,1/b$.
  Let~$\xi$ be the constant given by Lemma~\ref{lem:abs} and put $c'\gg c,$ $1/\xi$, where $c$ is the constant given by Lemma~\ref{lem:absa}.

  Suppose $p\geq2c'n^{-1/d^*(F)}$ and let $H$ be an $n$-vertex $k$-graph with vertex set $V$ such that $\delta(H)\geq (\a_F+\eps)\binom{n-1}{k-1}$.
  We will expose $H':=H^{(k)}(n,p)$ in two rounds: $H'=H_1\cup H_2$, where $H_1$ and $H_2$ are independent copies of $H^{(k)}(n,p')$ and $(1-p')^2=1-p$.  Note that since $(1-p')^2>1-2p'$, we have $p'> p/2\geq c'n^{-1/d^*(F)}$.

  The first step is to find an absorbing set in~$H\cup H_1$ using
  Lemma~\ref{lem:abs}.  We shall apply that lemma
  with~$V_0=\emptyset$.  By Lemma~\ref{lem:absa}, a.a.s., for every
  $S\in\binom{V}{b}$, there are at least~$\r n$ pairwise
  vertex-disjoint simple $S$-absorbers in $V\setminus S$ in
  $H\cup H_1$ and hence hypothesis~(\textit{i}) of Lemma~\ref{lem:abs}
  holds.  We now note that our simple absorbers have the following
  additional property.  Suppose~$A'$ is a simple $S$-absorber.  Then
  \begin{itemize}
  \item[(*)] $(H\cup H_1)[S\cup A']$ contains an $F$-factor
    $\{F_1,F_2,\dots\}$ with each $F_i$ with $|V(F_i)\cap S|\leq1$.
  \end{itemize}
  It follows from this property that the $\r n$ vertex-disjoint simple
  $S$-absorbers in $V\setminus S$ in $H\cup H_1$ force the existence
  of $\r n$ copies of~$F$ in $H\cup H_1$ that are pairwise
  vertex-disjoint except for the vertex~$v$.  Thus
  hypothesis~(\textit{ii}) of Lemma~\ref{lem:abs} holds.
  Lemma~\ref{lem:abs} implies that $H\cup H_1$ contains an absorbing
  set~$A$ of size at most~$\r n$~a.a.s.

  The second step is to find an almost $F$-factor in $H_2$ that covers all of $V\setminus A$ but at most~$\xi n$ vertices.
  Since $p'> p/2\geq c'n^{-1/d^*(F)}$, Lemma \ref{lem:PhiF} guarantees that $\Phi_F(n,p')\geq c'n$.
  Applying Lemma~\ref{lem:BHM}~\ref{BHMi} to $H_2$ with $\l=\xi$, a.a.s.\ every induced subgraph of $H_2$ of order~$\xi n$ contains a copy of $F$.
  Thus we can greedily find pairwise vertex-disjoint copies of $F$ in $V\setminus A$ until there are at most $\xi n$ vertices left.
  Denote by $R$ the set of these remaining vertices of $V\setminus A$.
  Then $|R|\leq\xi n$ a.a.s.

  The last step is to absorb all vertices of $R$ using the absorbing set $A$.
  Since $b=|V(F)|$ divides~$n$ and $V\setminus(A\cup R)$ is covered by vertex disjoint copies of $F$, we have that~$b$ divides $|A|+|R|$.
  Recall that~$A$ is an absorbing set in~$H\cup H_1$, which implies that $(H\cup H_1)[A\cup R]$ contains an $F$-factor.
  This implies that, a.a.s., we have the desired $F$-factor in~$H\cup H_1\cup H_2$.
\end{proof}

\section{Concluding Remarks}
\label{sec:concluding}
We have studied the $F$-factor problem in the randomly perturbed $k$-graphs $H\cup H^{(k)}(n,p)$ under the vertex degree condition $\delta(H)\geq (\a_F+\eps)\binom{n-1}{k-1}$ and edge probability condition~$p=\Omega(n^{-1/d^*(F)})$.
Let us close with a remark concerning the parameter~$\alpha_F$.
Recall~\eqref{eq:3} and recall that the definition of~$\Lambda_{F,v}$ involves the object~$F_v$.

Instead of~$F_v$, we can consider the \emph{link} $F_v':=\{e\setminus \{v\}: e\in E(F_v)\}$ of~$v$ in~$F$.
Then $\alpha_F=\min_{v\in V(F)}\pi(F_v')$, where~$\pi(F_v')$ is the \emph{Tur\'an density threshold} of the $(k-1)$-graph~$F_v'$, namely, $\pi(F_v')=\lim_{n\to\infty}\ex(n,F_v'){n\choose k-1}^{-1}$.
The fact that~$\alpha_F$ can be expressed in this way follows from supersaturation~\cite[Corollary 2]{Erdos}.
A related remark is that, while we require~$\eps' n^{v_F-1}$ embeddings of $F_v$ in our definition of~$\alpha_F$, supersaturation tells us that we could
require just one.

To conclude, we leave it as an open question whether or not~$\alpha_F$ is required (when $\alpha_F>0$) in the vertex degree condition $\delta(H)\geq (\a_F+\eps)\binom{n-1}{k-1}$.

\section{Acknowledgements}

We thank the anonymous referees for their careful reading and helpful comments.

\end{document}